\documentclass[12pt]{article}
\usepackage{amsmath}
\usepackage{amssymb}
\usepackage{amsthm}
\usepackage{extarrows}

\newtheorem{theorem}{Theorem}

\newtheorem{lemma}{Lemma}
\newtheorem{definition}{Definition}

\newtheorem{proposition}{Proposition}

\newcommand{\Mk}{\mathcal{M}_k}
\newcommand{\Mkk}{\mathcal{M}_{k-1}}
\newcommand{\Hk}{\mathcal{H}_k}
\newcommand{\Rm}{\mathbb{R}^m}
\newcommand{\C}{\mathbb{C}}
\newcommand{\Clm}{\mathcal{C}l_m}
\newcommand{\Ho}{\mathcal{H}_1}
\newcommand{\scs}{\mathcal{S}}
\newcommand{\Mo}{\mathcal{M}_1}

\newcommand{\puj}{\partial_{u_j}}
\newcommand{\udx}{\langle u,D_x\rangle}
\newcommand{\dudx}{\langle D_u,D_x\rangle}

\newcommand{\Dtwo}{\mathcal{D}_2}
\newcommand{\Dt}{\mathcal{D}_{3}}
\newcommand{\Df}{\mathcal{D}_{4}}

\begin{document}

\title{Third order fermionic and fourth order bosonic operators}

\author{Chao Ding\thanks{Electronic address:  {\tt dchao@uark.edu}.} $^1$ and Raymond Walter$^{1,2}$\thanks{Electronic address:  {\tt rwalter@email.uark.edu}; R.W. acknowledges this material is based upon work supported by the National Science Foundation Graduate Research Fellowship Program under Grant No. DGE-0957325 and the University of Arkansas Graduate School Distinguished Doctoral Fellowship in Mathematics and Physics.} \\
\emph{\small $^1$Department of Mathematics, University of Arkansas, Fayetteville, AR 72701, USA} \\ 
\emph{\small $^2$Department of Physics, University of Arkansas, Fayetteville, AR 72701, USA}\\
\\
\emph{\normalsize Dedicated to Professor John Ryan on the occasion of his 60th birthday}
}
\date{}

\maketitle

\begin{abstract}
This paper continues the work of our previous paper \cite{Ding1}, where we generalize $k$th-powers of the Euclidean Dirac operator $D_x$ to higher spin spaces in the case the target space is a degree one homogeneous polynomial space. In this paper, we reconsider the generalizations of $D_x^3$ and $D_x^4$ to higher spin spaces in the case the target space is a degree $k$ homogeneous polynomial space. Constructions of $3rd$ and $4th$ order conformally invariant operators in higher spin spaces are given; these are the $3rd$ order fermionic and $4th$ order bosonic operators. Fundamental solutions and intertwining operators of both operators are also presented here. These results can be easily generalized to cylinders and Hopf manifolds as in \cite{Ding}.
\end{abstract}

\section{Introduction}\hspace*{\fill} 
\par

The \emph{higher spin theory} in Clifford analysis began with the Rarita-Schwinger operator \cite{B}, which is named analogously to the Dirac operator and reproduces the wave equations for a massless particle of arbitrary half-integer spin in four dimensions with appropriate signature \cite{Ro}. The former operator takes its name from the 1941 work of Rarita and Schwinger \cite{Ra} that simply formulated the theory of particles of arbitrary half-integer spin $k+\frac{1}{2}$ and in particular considered its implications for particles of spin $\frac{3}{2}$. The higher spin theory considers generalizations of classical Clifford analysis techniques to higher spin spaces \cite{B1, Br1, B, D, E, Li}, focusing on operators acting on functions on $\Rm$ that take values in arbitrary irreducible representations of $Spin(m)$. Generally these are polynomial representations, such as $k$-homogeneous monogenic (harmonic) polynomials corresponding to particles of half-integer spin (integer spin). The highest weight vector of the spin representation as a whole may even be taken as a parameter \cite{HighestWeightExample}, but we consider a narrower scope.\\

\par
Slov\'{a}k \cite{J} provided a non-constructive classification of all conformally invariant differential operators on locally conformally flat manifolds in higher spin theory, but this shows only between which vector bundles these operators exist and what is their order; explicit expressions of these operators are still being found. Eelbode and Roels \cite{E} noted the Laplace operator $\Delta_x$ is no longer conformally invariant when acting on $C^{\infty}(\Rm,\Ho)$, where $\Ho$ is the degree one homogeneous harmonic polynomial space (correspondingly $\Mo$ for monogenic polynomials). They construct a second order conformally invariant operator on $C^{\infty}(\Rm,\Ho)$, the (generalized) Maxwell operator, reproducing the Maxwell equation for appropriate dimension and signature \cite{E}. De Bie and his co-authors \cite{B1} generalize this Maxwell operator from $C^{\infty}(\Rm,\Ho)$ to $C^{\infty}(\Rm,\Hk)$ to provide the higher spin Laplace operators, which are the second order conformally invariant operators generalizing the Laplace operator to arbitrary integer spins. Our earlier work \cite{Ding1} generalizes $D_x^k$ in higher spin spaces in the case the target space is a degree one homogeneous polynomial space, encompassing the spin-1 and spin-$\frac{3}{2}$ cases. 
In this paper, we consider 3rd-order fermionic and 4th-order bosonic operators corresponding to the appropriate degree-$k$ homogeneous polynomial space ($\Mk$ or $\Hk$). While \cite{Ding1} considers arbitrary order operators of lowest spin, this work considers arbitrary spin operators of $3rd$ and $4th$ order.\\
\par
The paper is organized as follows: We briefly introduce Clifford algebras, Clifford analysis, and representation theory of the Spin group in Section 2. In Section 3, we introduce the $3rd$-order higher spin operators $\Dt$ as the generalization of $D_x^3$ when acting on $C^{\infty}(\Rm, \Mk)$ and $4th$-order higher spin operators $\Df$ as the generalization of $D_x^4$ when acting on $C^{\infty}(\Rm, \Hk)$. Nomenclature for general higher order higher spin operators is given: bosonic and fermionic operators. The construction and conformal invariance of both operators are given with the help of the concept of \emph{generalized symmetry} as in \cite{B1,Ding1,E}. Then we provide the intertwining operators for $\Dt$ and $\Df$ with similar techniques as in \cite{Ding1}, which also reveal that these operators are conformally invariant. Section 4 presents the fundamental solutions and intertwining operators of  $\Dt$ and $\Df$ using similar techniques as in \cite{Ding1}. The expressions of the fundamental solutions also suggest that $\Dt$ and $\Df$ are generalizations of $D_x^3$ and $D_x^4$ in higher spin spaces and these can be generalized to conformally flat manifolds, for instance, cylinders and Hopf manifolds, as in \cite{Ding}.
\section*{Acknowledgement}
The authors are grateful to Bent \O rsted for communications pointing out that the intertwining operators of our conformally invariant differential operators can be recovered as Knapp-Stein intertwining operators in higher spin theory.
\section{Preliminaries}
\subsection{Clifford algebra}\hspace*{\fill}
A real Clifford algebra, $\Clm,$ can be generated from $\mathbb{R}^m$ by considering the
relationship $$\underline{x}^{2}=-\|\underline{x}\|^{2}$$ for each
$\underline{x}\in \mathbb{R}^m$.  We have $\mathbb{R}^m\subseteq \Clm$. If $\{e_1,\ldots, e_m\}$ is an orthonormal basis for $\mathbb{R}^m$, then $\underline{x}^{2}=-\|\underline{x}\|^{2}$ tells us that $$e_i e_j + e_j e_i= -2\delta_{ij},$$ where $\delta_{ij}$ is the Kronecker delta function. An arbitrary element of the basis of the Clifford algebra can be written as $e_A=e_{j_1}\cdots e_{j_r},$ where $A=\{j_1, \cdots, j_r\}\subset \{1, 2, \cdots, m\}$ and $1\leq j_1< j_2 < \cdots < j_r \leq m.$
Hence for any element $a\in \Clm$, we have $a=\sum_Aa_Ae_A,$ where $a_A\in \mathbb{R}$. Similarly, the complex Clifford algebra $\Clm (\C)$ is defined as the complexification of the real Clifford algebra
$$\Clm (\C)=\Clm\otimes\C.$$
We consider real Clifford algebra $\Clm$ throughout this subsection, but in the rest of the paper we consider the complex Clifford algebra $\Clm (\mathbb{C})$ unless otherwise specified. \\
\par
The Pin and Spin groups play an important role in Clifford analysis. The Pin group can be defined as $$Pin(m)=\{a\in \mathcal{C}l_m: a=y_1y_2\dots y_{p},\  y_1,\dots,y_{p}\in\mathbb{S}^{m-1},p\in\mathbb{N}\},$$ 
where $\mathbb{S} ^{m-1}$ is the unit sphere in $\Rm$. $Pin(m)$ is clearly a group under multiplication in $\mathcal{C}l_m$. \\
\par
Now suppose that $a\in \mathbb{S}^{m-1}\subseteq \mathbb{R}^m$, if we consider $axa$, we may decompose$x=x_{a\parallel}+x_{a\perp},$
where $x_{a\parallel}$ is the projection of $x$ onto $a$ and $x_{a\perp}$ is the rest, perpendicular to $a$. Hence $x_{a\parallel}$ is a scalar multiple of $a$ and we have $axa=ax_{a\parallel}a+ax_{a\perp}a=-x_{a\parallel}+x_{a\perp}.$
So the action $axa$ describes a reflection of $x$ in the direction of $a$. By the Cartan-Dieudonn$\acute{e}$ Theorem each $O\in O(m)$ is the composition of a finite number of reflections. If $a=y_1\cdots y_p\in Pin(m),$ we define $\tilde{a}:=y_p\cdots y_1$ and observe that $ax\tilde{a}=O_a(x)$ for some $O_a\in O(m)$. Choosing $y_1,\ \dots,\ y_p$ arbitrarily in $\mathbb{S}^{m-1}$, we see that the group homomorphism
\begin{eqnarray}
\theta:\ Pin(m)\longrightarrow O(m)\ :\ a\mapsto O_a,
\end{eqnarray}
with $a=y_1\cdots y_p$ and $O_ax=ax\tilde{a}$ is surjective. Further $-ax(-\tilde{a})=ax\tilde{a}$, so $1,\ -1\in Ker(\theta)$. In fact $Ker(\theta)=\{1,\ -1\}$. See \cite{P1}. The Spin group is defined as
$$Spin(m)=\{a\in \mathcal{C}l_m: a=y_1y_2\dots y_{2p},? y_1,\dots,y_{2p}\in\mathbb{S}^{m-1},p\in\mathbb{N}\}$$
 and it is a subgroup of $Pin(m)$. There is a group homomorphism
\begin{eqnarray*}
\theta:\ Spin(m)\longrightarrow SO(m)\ ,
\end{eqnarray*}
which is surjective with kernel $\{1,\ -1\}$. It is defined by $(1)$. Thus $Spin(m)$ is the double cover of $SO(m)$. See \cite{P1} for more details.\\
\par
For a domain $U$ in $\Rm$, a diffeomorphism $\phi: U\longrightarrow \mathbb{R}^m$ is said to be conformal if, for each $x\in U$ and each $\mathbf{u,v}\in TU_x$, the angle between $\mathbf{u}$ and $\mathbf{v}$ is preserved under the corresponding differential at $x$, $d\phi_x$. For $m\geq 3$, a theorem of Liouville tells us the only conformal transformations are M\"obius transformations. Ahlfors and Vahlen show that given a M\"{o}bius transformation on $\mathbb{R}^m \cup \{\infty\}$ it can be expressed as $y=(ax+b)(cx+d)^{-1}$ where $a,\ b,\ c,\ d\in \Clm$ and satisfy the following conditions \cite{Ah}:
\begin{eqnarray*}
&&1.\ a,\ b,\ c,\ d\ are\ all\ products\ of\ vectors\ in\ \mathbb{R}^m;\\
&&2.\ a\tilde{b},\ c\tilde{d},\ \tilde{b}c,\ \tilde{d}a\in\mathbb{R}^m;\\
&&3.\ a\tilde{d}-b\tilde{c}=\pm 1.
\end{eqnarray*}
Since $y=(ax+b)(cx+d)^{-1}=ac^{-1}+(b-ac^{-1}d)(cx+d)^{-1}$, a conformal transformation can be decomposed as compositions of translation, dilation, reflection and inversion. This gives an \emph{Iwasawa decomposition} for M\"obius transformations. See \cite{Li} for more details.
\par
The Dirac operator in $\mathbb{R}^m$ is defined to be $$D_x:=\sum_{i=1}^{m}e_i\partial_{x_i}.$$  Note $D_x^2=-\Delta_x$, where $\Delta_x$ is the Laplacian in $\mathbb{R}^m$.  A $\Clm$-valued function $f(x)$ defined on a domain $U$ in $\Rm$ is left monogenic if $D_xf(x)=0.$ Since multiplication of Clifford numbers are not commutative in general, there is a similar definition for right monogenic functions. Sometimes we will consider the Dirac operator $D_u$ in vector $u$ rather than $x$.\\
\par 
Let $\mathcal{M}_k$ denote the space of $\mathcal{C}l_m$-valued monogenic polynomials, homogeneous of degree $k$. Note that if $h_k\in\Hk$, the space of $\mathcal{C}l_m$-valued harmonic polynomials homogeneous of degree $k$, then $D_uh_k\in\mathcal{M}_{k-1}$, but $D_uup_{k-1}(u)=(-m-2k+2)p_{k-1}u,$ so
$$\mathcal{H}_k=\mathcal{M}_k\oplus u\mathcal{M}_{k-1},\ h_j=p_k+up_{k-1}.$$
This is an \emph{Almansi-Fischer decomposition} of $\Hk$. See \cite{D} for more details. In this Almansi-Fischer decomposition, we define $P_k$ as the projection map 
\begin{eqnarray*}
P_k: \mathcal{H}_k\longrightarrow \mathcal{M}_k.
\end{eqnarray*}
Suppose $U$ is a domain in $\mathbb{R}^m$. Consider a differentiable function $f: U\times \mathbb{R}^m\longrightarrow \mathcal{C}l_m,$
such that for each $x\in U$, $f(x,u)$ is a left monogenic polynomial homogeneous of degree $k$ in $u$, then the Rarita-Schwinger operator \cite{B,D} is defined by 
 $$R_kf(x,u):=P_kD_xf(x,u)=(\frac{uD_u}{m+2k-2}+1)D_xf(x,u).$$

\subsection{Irreducible representations of the Spin group}
The following three representation spaces of the Spin group are frequently used as the target spaces in Clifford analysis. The spinor representation is the most commonly used spin representation in classical Clifford analysis and the other two polynomial representations are often used in higher spin theory.
\subsubsection{Spinor representation of $Spin(m)$}
Consider the complex Clifford algebra $\mathcal{C}l_m(\mathbb{C})$ with even dimension $m=2n$. Then $\C^m$ or the space of vectors is embedded in $\Clm(\C)$ as
\begin{eqnarray*}
(x_1,x_2,\cdots,x_m)\mapsto \sum^{m}_{j=1}x_je_j:\ \C ^m\hookrightarrow \mathcal{C}l_m(\C).
\end{eqnarray*}
Define the \emph{Witt basis} elements of $\C^{2n}$ as 
$$f_j:=\displaystyle\frac{e_j-ie_{j+n}}{2},\ \ f_j^{\dagger}:=-\displaystyle\frac{e_j+ie_{j+n}}{2}.$$
Let $I:=f_1f_1^{\dagger}\dots f_nf_n^{\dagger}$. The space of \emph{Dirac spinors} is defined as
$$\mathcal{S}:=\Clm(\C)I.$$ This is a representation of $Spin(m)$ under the following action
$$\rho(s)I:=sI,\ for\ s\in Spin(m).$$
Note that $\scs$ is a left ideal of $\Clm (\C)$. For more details, we refer the reader to \cite{De}. An alternative construction of spinor spaces is given in the classical paper of Atiyah, Bott and Shapiro \cite{At}.
\subsubsection{Homogeneous harmonic polynomials on $\mathcal{H}_k(\Rm,\mathbb{C})$}
The space of harmonic polynomials is invariant under the action of $Spin(m)$ because the Laplacian $\Delta_m$ is an $SO(m)$-invariant operator, but this space it is not irreducible for $Spin(m)$, decomposing into the infinite sum of spaces of $k$-homogeneous harmonic polynomials, $0\leq k<\infty$, each of which is irreducible for $Spin(m)$. This brings us to a familiar representation of $Spin(m)$, that is $\mathcal{H}_k$. The following action has been shown to be an irreducible representation of $Spin(m)$ \cite{L}: 
\begin{eqnarray*}
\rho\ :\ Spin(m)\longrightarrow Aut(\Hk),\ s\longmapsto \big(f(x)\mapsto f(sy\tilde{s})\big).
\end{eqnarray*}
with $x=sy\tilde{s}$. This can also be realized as follows
\begin{eqnarray*}
Spin(m)\xlongrightarrow{\theta}SO(m)\xlongrightarrow{\rho} Aut(\Hk);\\
a\longmapsto O_a\longmapsto \big(f(x)\mapsto f(O_ax)\big),
\end{eqnarray*}
where $\theta$ is the double covering map and $\rho$ is the standard action of $SO(m)$ on a function $f(x)\in\Hk$ with $x\in\mathbb{R}^m$. The function $\phi(z)=(z_1+iz_m)^k$ is the highest weight vector for $\Hk (\Rm,\C)$ having highest weight $(k,0,\cdots,0)$ (for more details, see \cite{G}). Accordingly, spin representations given by $\mathcal{H}_k(\Rm,\mathbb{C})$ are said to have integer spin $k$; we can either specify an integer spin $k$ or degree of homogeneity $k$ of harmonic polynomials.

\subsubsection{Homogeneous monogenic polynomials on $\mathcal{C}l_m$}\hspace*{\fill}
In $\mathcal{C}l_m$-valued function theory, the previously mentioned Almansi-Fischer decomposition shows that we can also decompose the space of $j$-homogeneous harmonic polynomials as follows
$$\Hk=\Mk\oplus u\Mkk.$$
If we restrict $\Mk$ to the spinor valued subspace, we have another important representation of $Spin(m)$: the space of $j$-homogeneous spinor-valued monogenic polynomials on $\Rm$, henceforth denoted by $\Mk:=\Mk(\Rm,\mathcal{S})$. More specifically, the following action has been shown to be an irreducible representation of $Spin(m)$:
\begin{eqnarray*}
\pi\ :\ Spin(m)\longrightarrow Aut(\Mk),\ s\longmapsto (f(x)\mapsto sf(sx\tilde{s})).
\end{eqnarray*}
When $m$ is odd, in terms of complex variables $z_s=x_{2s-1}+ix_{2s}$ for all $1\leq s\leq \frac{m-1}{2}$, the highest weight vector is
$\omega_k(x)=(\bar{z_1})^kI$ for $\Mk(\Rm,\mathcal{S})$ having highest weight $(k+\frac{1}{2},\frac{1}{2},\cdots,\frac{1}{2})$, where $\bar{z_1}$ is the conjugate of $z_1$, $\mathcal{S}$ is the Dirac spinor space, and $I$ is defined as in Section $2.2.1$; for details, see \cite{L}. Accordingly, the spin representations given by $\mathcal{M}_k(\Rm,\mathcal{S})$ are said to have half-integer spin $k+\frac{1}{2}$; we can either specify a half-integer spin $k+\frac{1}{2}$ or the degree of homogeneity $k$ of monogenic spinor-valued polynomials. 
\par


\section{Construction and conformal invariance}
Slov\'{a}k \cite{J} established the existence of conformally invariant differential operators of arbitrary order and spin, provided that operators of odd order (respectively even order) have half integer spin $k+\frac{1}{2}$ (integer spin $k$) and are between spaces of $k$-homogeneous monogenic polynomials $\mathcal{M}_k$ (harmonic polynomials $\mathcal{H}_k$), more details can be found in \cite{Ding1}. The spin-$\frac{1}{2}$ and spin-0 cases are well established to arbitrary order: these are the powers of the Dirac and Laplace operators. We recently established the cases of spin-$\frac{3}{2}$ and spin-1 to arbitrary order. In the first order case for arbitrary (half integer) spin, the explicit form of the operator is well known: the Rarita-Schwinger operators. Preceding our work, Eelbode and Roels followed by De Bie et al. worked out the second order case for arbitrary (integer) spin in the generalized Maxwell operator and higher spin Laplace operators. We push further here, working out the third and fourth order cases for arbitrary spin: in our terminology, these are the 3rd order fermionic operators and 4th order bosonic operators. Our nomenclature emphasizes the motivation by mathematical physics: particles of half-integer spin are known as fermions and particles of integer spin are known as bosons, so the operators of half-integer spin take the name fermionic operators and those of integer spin take the name bosonic operators. 

\subsection{3rd order higher spin operator $\mathcal{D}_{3}$}
Our main result in the $3rd$ order higher spin case is the following theorem.
\begin{theorem}
Up to a multiplicative constant, the unique $3$rd-order conformally invariant differential operator is $ \mathcal{D}_{3,k}:C^{\infty}(\mathbb{R}^m,\mathcal{M}_k)\longrightarrow C^{\infty}(\mathbb{R}^m,\mathcal{M}_k)$, where
\begin{eqnarray*}
\mathcal{D}_{3}&=&D_x^3+\displaystyle\frac{4}{m+2k}\langle u, D_x\rangle\langle D_u, D_x\rangle D_x-\displaystyle\frac{4||u||^2\langle D_u, D_x\rangle^2D_x}{(m+2k)(m+2k-2)}-\displaystyle\frac{2u\langle D_u, D_x\rangle D_x^2}{m+2k}\\
&&-\displaystyle\frac{8u\langle u, D_x\rangle\langle D_u,D_x\rangle^2}{(m+2k)(m+2k-2)}-\displaystyle\frac{8u^3\langle D_u,D_x\rangle^3}{(m+2k)(m+2k-2)(m+6k-10)}.
\end{eqnarray*}
\end{theorem}
Hereafter we may suppress the $k$ index for the operator since there is little risk of confusion. Note the target space $\Mk$ is a function space, so any element in $C^{\infty}(\Rm,\Mk)$ has the form $f(x,u)\in \Mk$ for each fixed $x\in\Rm$ and $x$ is the variable on which $\Dt$ acts.\\
\par
Our proof of conformal invariance of this operator follows closely the method of \cite{E,Ding1}. In order to explain what conformal invariance means, we begin with the concept of a generalized symmetry (see for instance \cite{Eastwood}):
\begin{definition}
An operator $\eta_1$ is a generalized symmetry for a differential operator $\mathcal{D}$ if and only if there exists another operator $\eta_2$ such that $\mathcal{D}\eta_1=\eta_2\mathcal{D}$. Note that for $\eta_1=\eta_2$, this reduces to a definition of a (proper) symmetry: $\mathcal{D}\eta_1=\eta_1\mathcal{D}$.
\end{definition}
One determines the first order generalized symmetries of an operator, which span a Lie algebra \cite{E,Miller}. In this case, the first order symmetries will span a Lie algebra isomorphic to the conformal Lie algebra $\mathfrak{so}(1,m+1)$; in this sense, the operators we consider are conformally invariant. The operator $\Dt$ is $\mathfrak{so}(m)$-invariant (rotation-invariant) because it is the composition of $\mathfrak{so}(m)$-invariant (rotation-invariant) operators, which means the angular momentum operators $L_{ij}^x+L_{i,j}^u$ that generate these rotations are proper symmetries of $\Dt$.
 The infinitesimal translations $\partial_{x_j}, j=1,\cdots , n,$ corresponding to linear momentum operators are proper  symmetries of $\Dt$; this is an alternative way to say that $\Dt$ is invariant under translations that are generated by these infinitesimal translations. Readers familiar with quantum mechanics will recognize the connection to isotropy and homogeneity of space, the rotational and translational invariance of Hamiltonian, and the conservation of angular and linear momentum \cite{Sakurai}; see also \cite{Br} concerning Rarita-Schwinger operators. 
 
 The remaining two of the first order generalized symmetries of $\Dt$ are the Euler operator and special conformal transformations. The Euler operator $\mathbb{E}_x$ that measures degree of homogeneity in $x$ is a generalized symmetry because $\Dt\mathbb{E}_x=(\mathbb{E}_x+3)\mathcal{D}_{3}$; this is an alternative way to say that $\Dt$ is invariant under dilations, which are generated by the Euler operator. The special conformal transformations are defined in Lemma \ref{SCT} in terms of harmonic inversion for $\Ho$-valued functions; harmonic inversion is defined in Definition \ref{HI} and is an involution mapping solutions of $\Dt$ to $\Dt$. Readers familiar with conformal field theory will recognize that invariance under dilation corresponds to scale-invariance and that special conformal transformations are another class of conformal transformations arising on spacetime \cite{CFT}. An alternative method of proving conformal invariance of $\Dt$ is to prove the invariance of $\Dt$ under those finite transformations generated by these first order generalized symmetries (rotations, dilations, translations, and special conformal transformations) to show invariance of $\Dt$ under actions of the conformal group; this may be phrased in terms of M\"obius transformations and the Iwasawa decomposition. However, the first-order generalized symmetry method emphasizes the connection to mathematical physics and is more amenable to our proof of a certain property of harmonic inversion. It is also that used by earlier authors \cite{B1,E}. 

\begin{definition} \label{HI}
The monogenic inversion is a conformal transformation defined as
\begin{eqnarray*}
\mathcal{J}_{3}:C^{\infty}(\mathbb{R}^m,\mathcal{M}_k)\longrightarrow C^{\infty}(\mathbb{R}^m,\mathcal{M}_k):f(x,u)\mapsto \mathcal{J}_{3}[f](x,u):=\frac{x}{||x||^{m-2}}f(\frac{x}{||x||^2},\frac{xux}{||x||^2}).
\end{eqnarray*}
\end{definition}
Note that this inversion consists of Kelvin inversion $\mathcal{J}$ on $\mathbb{R}^m$ in the variable $x$ composed with a reflection $u\mapsto \omega u\omega$ acting on the dummy variable $u$ (where $x=||x||\omega$) and a multiplication by a conformal weight term $\displaystyle\frac{x}{||x||^{m-2}}$; it satisfies $\mathcal{J}_{3}^2=-1.$\\
\par
Then we have the special conformal transformation defined in the following lemma. The definition is an infinitesimal version of the fact that finite special conformal transformations consist of a translation preceded and followed by an inversion \cite{CFT}: an infinitesimal translation preceded and followed by monogenic inversion. The second equality in the lemma shares some terms in common with the generators of special conformal transformations in conformal field theory \cite{CFT}, and is a particular case of a result in \cite{ER}.

\begin{lemma} \label{SCT}
The special conformal transformation defined as $\mathcal{C}_{3}:=\mathcal{J}_{3}\partial_{x_j}\mathcal{J}_{3}$ satisfies
\begin{eqnarray*}
\mathcal{J}_{3}\partial_{x_j}\mathcal{J}_{3}=xe_j-2\langle u,x\rangle\partial_{u_j}+2u_j\langle x,D_u\rangle -||x||^2\partial_{x_j}+x_j(2\mathbb{E}_x+m-2).
\end{eqnarray*}
\end{lemma}
\begin{proof}
A similar calculation as in \emph{Proposition A.1} in \cite{B1} will show the conclusion.
\end{proof}
Then, we have the main proposition as follows.
\begin{proposition}\label{propthree}
The special conformal transformations $\mathcal{C}_{3}$, with $j\in\{1,2,\dots,m\}$ are generalized symmetries of $\Dt$. More specifically,
\begin{eqnarray*}
[\Dt,\mathcal{C}_{3}]=6x_j\Dt.
\end{eqnarray*}
 In particular, this shows that 
\begin{eqnarray}\label{crucialeven}
\mathcal{J}_{3}\Dt\mathcal{J}_{3}=||x||^{6}\Dt,
\end{eqnarray}
 which is the generalization of $D_x^3$ in classical Clifford analysis \cite{P}. This also implies $\Dt$ is invariant under inversion.
 \end{proposition}
 If the main proposition holds, then the conformal invariance can be summarized in the following theorem:
 \begin{theorem}\label{theoremeven}
 The first order generalized symmetries of $\Dt$ are given by:
 \begin{enumerate}
 \item The infinitesimal rotation $L_{i,j}^x+L_{i,j}^u$, with $1\leq i<j\leq m$.
 \item The shifted Euler operator $(\mathbb{E}_x+\displaystyle\frac{m-2}{2})$.
 \item The infinitesimal translations $\partial_{x_j}$, with $1\leq j\leq m$.
 \item The special conformal transformations $\mathcal{J}_{3}\partial_{x_j}\mathcal{J}_{3}$, with $1\leq j\leq m$.
 \end{enumerate}
 These operators span a Lie algebra which is isomorphic to the conformal Lie algebra $\mathfrak{so}(1,m+1)$, whereby the Lie bracket is the ordinary commutator.
 \end{theorem}
\begin{proof}
The proof is similar as in \cite{ER} via transvector algebras.
\end{proof}
\subsubsection*{Detailed proof of Proposition \ref{propthree}:}
To prove this proposition, we first  introduce the following technical lemmas:

\begin{lemma}\label{lemma 1}
For all $1 \leq j \leq m$, we have
\begin{eqnarray*}
&&[D_x^3,\mathcal{C}_{3}]=4\langle u, D_x\rangle D_x\partial_{u_j}-2u\partial_{u_j}D_x^2-4u_jD_x\langle D_u,D_x\rangle+6x_jD_x^3.
\end{eqnarray*}
\end{lemma}

\begin{lemma}\label{lemma 2}
For all $1 \leq j \leq m$, we have
\begin{eqnarray*}
&&[\langle u,D_x\rangle\langle D_u,D_x\rangle D_x, \mathcal{C}_{3}] =-(m+2k)\langle u,D_x\rangle D_x\partial_{u_j}-e_ju\langle D_u,D_x\rangle D_x\\
&&+(m+2k-2)u_j\langle D_u,D_x\rangle D_x-2u\langle u,D_x\rangle\langle D_u,D_x\rangle\partial_{u_j}-2|u|^2\langle D_u,D_x\rangle D_x\partial_{u_j}\\
&&+6x_j\langle u,D_x\rangle\langle D_u,D_x\rangle D_x.
\end{eqnarray*}
\end{lemma}

\begin{lemma}\label{lemma 3}
For all $1 \leq j \leq m$, we have
\begin{eqnarray*}
&&[|u|^2\dudx^2 D_x,\mathcal{C}_{3}]=2|u|^2\dudx^2e_j-(2m+4k-4)|u|^2\dudx D_x\puj\\
&&\ \ -2u|u|^2\dudx^2\puj+6x_j|u|^2\dudx^2D_x.
\end{eqnarray*}
\end{lemma}

\begin{lemma}\label{lemma 4}
For all $1 \leq j \leq m$, we have
\begin{eqnarray*}
&&[u\dudx D_x^2,\mathcal{C}_{3}]=-2e_ju\dudx D_x-4u_j\dudx D_x-(m+2k)uD_x^2\puj\\
&&\ \ +4u\udx\dudx\puj-4u_ju\dudx^2+6x_ju\dudx D_x^2.
\end{eqnarray*}
\end{lemma}

\begin{lemma}\label{lemma 5}
For all $1 \leq j \leq m$, we have
\begin{eqnarray*}
&&[u\udx\dudx^2,\mathcal{C}_{3}]=-e_j|u|^2\dudx^2-(2m+4k-4)u\udx\dudx\puj\\
&&\ \ -2u|u|^2\dudx^2\puj+(m+2k-2)u_ju\dudx^2+6x_ju\udx\dudx^2.
\end{eqnarray*}
\end{lemma}

\begin{lemma}\label{lemma 6}
For all $1 \leq j \leq m$, we have
\begin{eqnarray*}
&&[u^3\dudx^3,\mathcal{C}_{3}]=-(m+6k-10)u^3\dudx^2\puj+6x_ju^3\dudx^3.
\end{eqnarray*}
\end{lemma}
To prove these lemmas, we calculate the commutators of our operator and each component of $\mathcal{C}_3$, then combining them gives the results. We use these lemmas to obtain 
\begin{eqnarray*}
[\mathcal{D}_{3},\mathcal{C}_{3}]=6x_j\mathcal{D}_{3}.
\end{eqnarray*}
Similar arguments as in \cite{Ding1} give that $\mathcal{J}_{3}\mathcal{D}_{3}\mathcal{J}_{3}=||x||^{6}\mathcal{D}_{3}$, which can be rewritten as
$$\mathcal{D}_{3,y,w}\frac{x}{||x||^{m-2}}f(y,w)=\frac{x}{||x||^{m+2}}\mathcal{D}_{3,x,u}f(x,u),\ \forall f(x,u)\in C^{\infty}(\Rm,\Mk),$$
where $y=x^{-1}$ and $w=\displaystyle\frac{xux}{||x||^2}$. Therefore, we have proved $\mathcal{D}_{3}$ is invariant under inversion.

\subsection{4th order higher spin operator $\Df$}
Now for the main result in the $4th$ order higher spin case.
\begin{theorem}
Up to a multiplicative constant, the unique $4$th-order conformally invariant differential operator is $\mathcal{D}_{4}:C^{\infty}(\mathbb{R}^m,\mathcal{H}_k)\longrightarrow C^{\infty}(\mathbb{R}^m,\mathcal{H}_k)$, where
\begin{eqnarray*}
\mathcal{D}_{4}=\Dtwo^2-\displaystyle\frac{8}{(m+2k-2)(m+2k-4)}\Dtwo\Delta_x.
\end{eqnarray*}
\end{theorem}
Hereafter we may suppress the $k$ index for the operator since there is little risk of confusion. The strategy is similar to that used above. It is sufficient to show only invariance under inversion. We have the definition for harmonic inversion as follows.
\begin{definition}
Harmonic inversion is a (conformal) transformation defined as
\begin{eqnarray*}
\mathcal{J}_{4}:C^{\infty}(\mathbb{R}^m,\mathcal{H}_k)\longrightarrow C^{\infty}(\mathbb{R}^m,\mathcal{H}_k):f(x,u)\mapsto \mathcal{J}_{4}[f](x,u):=||x||^{4-m}f(\frac{x}{||x||^2},\frac{xux}{||x||^2}).
\end{eqnarray*}
\end{definition}
Note this inversion consists of the classical Kelvin inversion $\mathcal{J}$ on $\mathbb{R}^m$ in the variable $x$ composed with a reflection $u\mapsto \omega u\omega$ acting on the dummy variable $u$ (where $x=||x||\omega$) and a multiplication by a conformal weight term $||x||^{4-m}$. It satisfies $\mathcal{J}_{4}^2=1$. Then a similar calculation as in \emph{Proposition A.1} in \cite{B1} provides the following lemma.
\begin{lemma}
The special conformal transformation is defined as
\begin{eqnarray*}
\mathcal{C}_{4}:=\mathcal{J}_{4}\partial_{x_j}\mathcal{J}_{4}=2\langle u,x\rangle\partial_{u_j}-2u_j\langle x,D_u\rangle +||x||^2\partial_{x_j}-x_j(2\mathbb{E}_x+m-4).
\end{eqnarray*}
\end{lemma}
\begin{proposition}
The special conformal transformations $\mathcal{C}_{4}$, with $j\in\{1,2,\dots,m\}$ are generalized symmetries of $\mathcal{D}_{4}$. More specifically,
\begin{eqnarray*}
[\mathcal{D}_4,\mathcal{C}_{4}]=-8x_j\mathcal{D}_{4}.
\end{eqnarray*}
In particular, this shows $\mathcal{J}_{4}\mathcal{D}_{4}\mathcal{J}_{4}=||x||^{8}\mathcal{D}_{4}$, which generalizes the case of the classical higher order Dirac operator $D_x^{4}$. This also implies $\mathcal{D}_{4}$ is invariant under inversion and hence conformally invariant.
\end{proposition}
This proposition follows immediately with the help of the following two lemmas.
\begin{lemma}
\begin{eqnarray*}
&&\big[\Dtwo^2,\mathcal{C}_4\big]=-8x_j\Dtwo^2+\frac{32\langle u,D_x\rangle\Delta_x\partial_{u_j}}{(m+2k-2)^2}-\frac{32u_j\langle D_u,D_x\rangle\Delta_x}{(m+2k-2)^2}-\frac{128\langle u,D_x\rangle^2\langle D_u,D_x\rangle\partial_{u_j}}{(m+2k-2)^2(m+2k-4)}\\
&+&\frac{128||u||^2\langle D_u,D_x\rangle\Delta_x\partial_{u_j}}{(m+2k-2)^2(m+2k-4)^2}-\frac{128||u||^2\langle D_u,D_x\rangle^2\partial_{x_j}}{(m+2k-2)^2(m+2k-4)^2}+\frac{128u_j\langle u,D_x\rangle\langle D_u,D_x\rangle^2}{(m+2k-2)^2(m+2k-4)}\\
&+&\frac{128||u||^2\langle u,D_x\rangle\langle D_u,D_x\rangle^2\partial_{u_j}}{(m+2k-2)^2(m+2k-4)^2}-\frac{128u_j||u||^2\langle D_u,D_x\rangle^3}{(m+2k-2)^2(m+2k-4)^2}.
\end{eqnarray*}
\end{lemma}

\begin{lemma}
\begin{eqnarray*}
&&\big[\Dtwo\Delta_x,\mathcal{C}_4\big]=-8x_j\Dtwo\Delta_x+\frac{4m+8k-16}{m+2k-2}\langle u,D_x\rangle\Delta_x\partial_{u_j}-\frac{16\langle u,D_x\rangle^2\langle D_u,D_x\rangle\partial_{u_j}}{m+2k-2}\\
&+&\frac{16||u||^2\langle D_u,D_x\rangle\Delta_x\partial_{u_j}}{(m+2k-2)(m+2k-4)}+\frac{16||u||^2\langle u,D_x\rangle\langle D_u,D_x\rangle^2\partial_{u_j}}{(m+2k-2)(m+2k-4)}-\frac{4m+8k-16}{m+2k-2}u_j\langle D_u,D_x\rangle\Delta_x\\
&+&\frac{16u_j\langle u,D_x\rangle\langle D_u,D_x\rangle^2}{m+2k-2}-\frac{16||u||^2\langle D_u,D_x\rangle^2\partial_{x_j}}{(m+2k-2)(m+2k-4)}-\frac{16u_j||u||^2\langle D_u,D_x\rangle^3}{(m+2k-2)(m+2k-4)}.
\end{eqnarray*}
\end{lemma}
\begin{proof}
With the help of $[AB,C]=A[B,C]+[A,C]B$, where $A,\ B$ and $C$ are operators, a straightforward calculation leads to the result that
 \begin{eqnarray*}
[\mathcal{D}_4,\mathcal{C}_{4}]=-8x_j\mathcal{D}_{4}.
\end{eqnarray*}
Similar arguments as in the 3rd order case complete the proof.
\end{proof}

\section{Connection with lower order conformally invariant operators}
To construct higher order conformally invariant operators, one possible method is by composing and combining lower order conformally invariant operators. In this section, we will rewrite our operators $\Dt$ and $\Df$ in terms of first order and second order conformally invariant operators. We expect this will help us when we consider other higher order conformally invariant operators. \\
\par
Recall $\Dt$ maps $C^{\infty}(\Rm,\Mk)$ to $C^{\infty}(\Rm,\Mk)$. If we fix $x\in\Rm$, then for any $f(x,u)\in\Mk$, we have $\Dt f(x,u)\in\Mk$. In other words, $\Dt$ should be equal to the sum of contributions  to $\Mk$ of all terms in $\Dt$.
 Notice that if we apply each term of $\Dt$ to $f(x,u)\in C^{\infty}(\Rm,\Mk)$, we will get a $k$-homogeneous polynomial in $u$ that is in the kernel of $\Delta_u^2$. Hence, we can decompose it by harmonic decomposition as follows
\begin{eqnarray*}
\mathcal{P}_k=\Hk\oplus u^2\mathcal{H}_{k-2}.
\end{eqnarray*}
where $\mathcal{P}_k$ is the $k$-homogeneous polynomial space and $\Hk$ is the $k$-homogeneous harmonic polynomial space. The Almansi-Fischer decomposition provides further
\begin{eqnarray*}
\Hk=\Mk\oplus u\Mkk,
\end{eqnarray*}
where $\Mk$ is the $k$-homogeneous monogenic polynomial space; therefore, the contribution of each term to $\Mk$ can be written with two projections. For instance, the contribution of $u^3\langle D_u,D_x\rangle^3f(x,u)$ to $\Mk$ is $P_kP_1u^3\langle D_u,D_x\rangle^3f(x,u)$, where
\begin{eqnarray*}
\mathcal{P}_k\xrightarrow{P_1}\Hk\xrightarrow{P_k}\Mk,
\end{eqnarray*} 
and
\begin{eqnarray*}
P_1=1+\frac{u^2\Delta_u}{2(m+2k-4)},\ P_k=1+\frac{uD_u}{m+2k-2}.
\end{eqnarray*}

We also notice that for fixed $x\in\Rm$ and $f(x,u)\in\Mk$,
$$u^3\langle D_u,D_x\rangle^3f(x,u),\ ||u||^2\langle D_u, D_x\rangle^2D_xf(x,u)\in u^2\mathcal{H}_{k-2},$$
and $u\langle D_u, D_x\rangle D_x^2\in u\Mkk$. Hence, their contributions to $\Mk$ are all zero.
Therefore,
\begin{eqnarray*}
\Dt=P_kP_1\bigg(D_x^3+\displaystyle\frac{4}{m+2k}\langle u, D_x\rangle\langle D_u, D_x\rangle D_x-\displaystyle\frac{8u\langle u, D_x\rangle\langle D_u,D_x\rangle^2}{(m+2k)(m+2k-2)}\bigg).
\end{eqnarray*}
It is useful to recall some first and second order conformally invariant operators in higher spin spaces \cite{B1,B}:
\begin{eqnarray*}
&&R_k:\ C^{\infty}(\Rm,\Mk)\longrightarrow C^{\infty}(\Rm,\Mk),\ 
R_k=P_kD_x=(1+\frac{uD_u}{m+2k-2})D_x;\\
&&T_k:\ C^{\infty}(\Rm,u\Mkk)\longrightarrow C^{\infty}(\Rm,\Mk),\ 
T_k=P_kD_x=(1+\frac{uD_u}{m+2k-2})D_x;\\
&&T_k^*:\ C^{\infty}(\Rm,\Mk)\longrightarrow C^{\infty}(\Rm,u\Mkk),\ 
T_k^*=(I-P_k)D_x=\frac{uD_u}{m+2k-2}D_x;\\
&&\Dtwo:\ C^{\infty}(\Rm,\Hk)\longrightarrow C^{\infty}(\Rm,\Hk),\ 
\Dtwo=P_1(\Delta_x-\frac{4}{m+2k-2}\langle u,D_x\rangle\langle D_u,D_x\rangle D_x).
\end{eqnarray*}
Hence,
\begin{eqnarray*}
\Dt=P_kP_1\bigg(D_x^3+\frac{4\langle u,D_x\rangle\langle D_u,D_x\rangle D_x}{m+2k-2}-\frac{8\langle u,D_x\rangle\langle D_u,D_x\rangle D_x}{(m+2k)(m+2k-2)}
-\frac{8u\langle u,D_x\rangle\langle D_u,D_x\rangle^2}{(m+2k)(m+2k-2)}\bigg)\\
=-P_kP_1\Dtwo D_x-\frac{8}{(m+2k)(m+2k-2)}P_kP_1\big(\langle u,D_x\rangle\langle D_u,D_x\rangle D_x+u\langle u,D_x\rangle\langle D_u,D_x\rangle^2\big).
\end{eqnarray*}

Since for $f(x,u)\in C^{\infty}(\Rm,\Mk)$, we have \cite{B1}:
$$\Dtwo=-R_k^2+\frac{4u\langle D_u,D_x\rangle}{(m+2k-2)(m+2k-4)}R_k.$$
A straightforward calculation leads to
$$\Dt=R_k^3+\displaystyle\frac{4}{(m+2k)(m+2k-4)}T_kT_k^*R_k.$$
\par
Recall these conformally invariant second order twistor and dual-twistor operators \cite{B1}:
\begin{eqnarray*}
	&&T_{k,2}=\langle u,D_x\rangle -\frac{||u||^2\langle D_u,D_x\rangle}{m+2k-4}:\ C^{\infty}(\Rm,\mathcal{H}_{k-1})\longrightarrow C^{\infty}(\Rm,\Hk),\\
	&&T_{k,2}^*=\langle D_u,D_x\rangle:\ C^{\infty}(\Rm,\Hk)\longrightarrow C^{\infty}(\Rm,\mathcal{H}_{k-1}), \text{ and}
\\&&\Dtwo=\Delta_x-\frac{4T_{k,2}T_{k,2}^*}{m+2k-2}.
\end{eqnarray*}
Hence
\begin{eqnarray*}
\Df&=&\Dtwo^2-\frac{8\Dtwo\Delta_x}{(m+2k-2)(m+2k-4)}\\
   &=&\Dtwo^2-\frac{8\Dtwo}{(m+2k-2)(m+2k-4)}\big(\Dtwo+\frac{4T_{k,2}T_{k,2}^*}{m+2k-2}\big)\\
   &=&\frac{(m+2k)(m+2k-6)}{(m+2k-2)(m+2k-4)}\Dtwo^2-\frac{32\Dtwo T_{k,2}T_{k,2}^*}{(m+2k-2)^2(m+2k-4)}.
\end{eqnarray*}
\section{Fundamental solutions and Intertwining operators}
Using similar arguments as in \cite{Ding1}, we obtain the fundamental solutions (up to a multiplicative constant) and intertwining operators of $\Dt$ and $\Df$ as follows.
\begin{theorem}\textbf{(Fundamental solutions of $\Dt$)}\\
Let $Z_{k}(u,v)$ be the reproducing kernel of $\Mk$, then the fundamental solutions of $\Dt$ are
$$c_1\displaystyle\frac{x}{||x||^{m-2}}Z_k(\displaystyle\frac{xux}{||x||^2},v),$$
where $c_1$ is a constant.
\end{theorem}
\begin{theorem}\textbf{(Fundamental solutions of $\Df$)}\\
Let $Z_{k}(u,v)$ be the reproducing kernel of $\Hk$, then the fundamental solutions of $\Df$ are
$$c_2||x||^{4-m}Z_k(\displaystyle\frac{xux}{||x||^2},v),$$
where $c_2$ is a constant.
\end{theorem}

\begin{theorem}\textbf{(Intertwining operators)}\\
Let $y=\phi(x)=(ax+b)(cx+d)^{-1}$ be a M\"{o}bius transformation. Then
\begin{eqnarray*}
\frac{\widetilde{cx+d}}{||cx+d||^{m+4}}\mathcal{D}_{3,y,\omega}f(y,\omega)=\mathcal{D}_{3,x,u}\frac{\widetilde{cx+d}}{||cx+d||^{m-2}}f(\phi(x),\frac{(cx+d)u\widetilde{(cx+d)}}{||cx+d||^2}),
\end{eqnarray*}
where $\omega=\displaystyle\frac{(cx+d)u\widetilde{(cx+d)}}{||cx+d||^2}$ and $f(y,\omega)\in C^{\infty}(\Rm,\Mk)$;
\begin{eqnarray*}
||cx+d||^{-m-4}\mathcal{D}_{4,y,\omega}f(y,\omega)=\mathcal{D}_{4,x,u}||cx+d||^{4-m}f(\phi(x),\frac{(cx+d)u\widetilde{(cx+d)}}{||cx+d||^2}),
\end{eqnarray*}
where $\omega=\displaystyle\frac{(cx+d)u\widetilde{(cx+d)}}{||cx+d||^2}$ and $f(y,\omega)\in C^{\infty}(\Rm,\Hk)$.
\end{theorem}

It is worth pointing out that our above results generalize to conformally flat manifolds according to the method in our paper on cylinders and Hopf manifolds \cite{Ding}.


\begin{thebibliography}{99}
\bibitem{Ah} L. V. Ahlfors, \emph{M\"{o}bius transformations in $\mathbb{R}^n$ expressed through 2$\times$2 matrices of Clifford numbers}, Complex Variables, 5, 1986, 215-224.

\bibitem{At} M.F. Atiyah, R. Bott, A. Shapiro, \emph{Clifford modules}, Topology, Vol. 3, Suppl. 1, 1964, pp.3-38.


\bibitem{B1} Hendrik De Bie, David Eelbode, Matthias Roels, \emph{The higher spin Laplace operator},  	arXiv:1501.03974 [math-ph]

\bibitem{Br} F. Brackx, R. Delanghe, F. Sommen, \emph{Clifford Analysis}, Pitman, London, 1982.

\bibitem{Br1} F. Brackx, D. Eelbode, L. Van de Voorde, \emph{Higher spin Dirac operators between spaces of simplicial monogenics in two vector variables}, Mathemal Physics, Analysis and Geometry, 2011, Vol. 14, Issue 1, pp. 1-20.

\bibitem{B} J. Bure\v{s}, F. Sommen, V. Sou\v{c}ek, P. Van Lancker, \emph{Rarita-Schwinger Type Operators in Clifford Analysis}, J. Funct. Anal. 185 (2001), No. 2, pp. 425-455.

\bibitem{De} R. Delanghe, F. Sommen, V. Sou\v{c}ek, \emph{Clifford Algebra and Spinor-Valued Functions: A Function Theory for the Dirac Operator}, Kluwer, Dordrecht, 1992.

\bibitem{Ding} Chao Ding, Raymond Walter, John Ryan, \emph{Higher order fermionic and bosonic operators on cylinders and Hopf manifolds}, http://arxiv.org/abs/1512.07323, submitted

\bibitem{Ding1} Chao Ding, Raymond Walter, John Ryan, \emph{Higher order fermionic and bosonic operators}, http://arxiv.org/abs/1512.07322, submitted

\bibitem{D} Charles F. Dunkl, Junxia Li, John Ryan and Peter Van Lancker, \emph{Some Rarita-Schwinger type operators}, Computational Methods and Function Theory, 2013, Vol. 13, Issue 3, pp. 397-424.

\bibitem{Eastwood} M. Eastwood, \emph{Higher symmetries of the Laplacian}, Ann. Math., 161 No. 3, 2005, pp. 1645-1665.

\bibitem{ER} D. Eelbode, T. Raeymaekers, \emph{Construction of conformally invariant higher spin operators using transvector algebras}, Journal of Mathematical Physics, Vol. 55, Issue 10, 2014, DOI:http://dx.doi.org/10.1063/1.4898772

\bibitem{E} David Eelbode, Matthias Roels, \emph{Generalised Maxwell equations in higher dimensions}, Complex Analysis and Operator Theory, Dec. 2014, pp. 1-27. dx.doi.org/10.1007/s11785-014-0436-5

\bibitem{CFT} P. Di Francesco, P. Mathieu, D. S\'en\'echal, \emph{Conformal Field Theory}, Graduate Texts in Contemporary Physics, New York: Springer-Verlag, 1997.


\bibitem{G} J. Gilbert and M. Murray, \emph{Clifford Algebras and Dirac Operators in Harmonic Analysis},  Cambridge University Press, Cambridge, 1991.


\bibitem{L} P. Van Lancker, F. Sommen, D. Constales, \emph{Models for irreducible representations of Spin(m)}, Advances in Applied Clifford Algebras, Vol. 11, Issue 1 supplement, 2001, pp. 271-289.


\bibitem{Li} Junxia, Li, John Ryan, \emph{Some operators associated to Rarita-Schwinger type operators}, Complex Variables and Elliptic Equations: An International Journal, Volume 57, Issue 7-8, 2012, pp. 885-902.

\bibitem{Miller} W. Miller, \emph{Symmetry and Separation of Variables}, Addison-Wesley Publishing Company, Providence, Rhode Island, 1977.


\bibitem{P} J. Peetre, T. Qian, \emph{M\"{o}bius covariance of iterated Dirac operators}, J. Austral. Math. Soc. Series A 56 (1994), pp. 403-414.

\bibitem{P1} I. Porteous, \emph{Clifford algebra and the classical groups}, Cambridge University Press, Cambridge, 1995.

\bibitem{Ra} W. Rarita, J. Schwinger, \emph{On a Theory of Particles with Half-integral Spin}, Phys. Rev., Vol. 60, Issue 1, 1941, pp. 60-61.

\bibitem{Ro} M. Roels, \emph{A Clifford analysis approach to higher spin fields}, Master's Thesis, University of Antwerp, 2013.


\bibitem{Sakurai} J.J. Sakurai, J. Napolitano, \emph{Modern Quantum Mechanics}, Second Edition, San Francisco: Addison-Wesley, 2011.

\bibitem{HighestWeightExample} H. De Schepper, D. Eelbode, T. Raeymaekers, \emph{On a special type of solutions of arbitrary higher spin Dirac operators}, J. Phys. A: Math. Theor., 43, 2010, 325208-325221.

\bibitem{J} J. Slov\'{a}k, \emph{Natural Operators on Conformal Manifolds}, Habilitation thesis, Masaryk University, Brno, Czech Republic, 1993.



\end{thebibliography}
\end{document}